\documentclass[11pt, reqno]{amsart}
\usepackage[margin=1in]{geometry}
\usepackage{graphicx}
\usepackage{relsize}
\usepackage{amsmath}
\usepackage{xcolor}
    \usepackage{braket}
\newcommand\numberthis{\addtocounter{equation}{1}\tag{\theequation}}

\numberwithin{equation}{section}
\theoremstyle{plain}
\newtheorem{thm}{Theorem}[section]
\newtheorem*{thm*}{Theorem}

\newtheorem{lem}[thm]{Lemma}

\newtheorem*{cor*}{Corollary}
\theoremstyle{definition}
\newtheorem{defn}{Definition}[section]

\theoremstyle{remark}
\newtheorem{rem}{Remark}[section]

\newcommand{\BigOq}[1]{\ensuremath{\operatorname{O}_q\left(#1\right)}}

\newcommand{\BigO}[1]{\ensuremath{\operatorname{O}\left(#1\right)}}

\makeatletter
\def\pmod#1{\allowbreak\mkern10mu({\operator@font mod}\,\,#1)} 
\makeatother

\begin{document}
\subjclass[2010]{11M06, 11N37, 11R42}
\keywords{Ramanujan sums, number fields, Perron formulas, Dedekind zeta function}

\title{On the distribution of Ramanujan Sums over number fields}

\author{Sneha Chaubey}
\email{sneha@iiitd.ac.in}

\author{Shivani Goel}
\email{shivanig@iiitd.ac.in}

\address[]{Department of Mathematics, IIIT Delhi, New Delhi 110020}
\begin{abstract}
   For a number field $\mathbb{K}$, and integral ideals $\mathcal{I}$ and $\mathcal{J}$ in its number ring $\mathcal{O}_{\mathbb{K}}$, Nowak studied the asymptotic behaviour of the average of Ramanujan sums $C_{\mathcal{J}}({\mathcal{I}})$ over both ideals $\mathcal{I}$ and $\mathcal{J}$. In this article, we extend this investigation by establishing asymptotic formulas for the second moment of averages of Ramanujan sums over quadratic and cubic number fields, thereby generalizing previous works of Chen, Kumchev, Robles, and Roy on moments of averages of Ramanujan sums over rationals. Additionally, using a special property of certain integral domains, we obtain second moment results for Ramanujan sums over some other number fields.
\end{abstract}
\maketitle
\section{Introduction and Main Results}

Ramanujan in 1918 \cite{ramanujan1918certain}, while studying the trigonometric series  representations  of normalized  arithmetic functions, introduced a function
\[c_n(m):=\sum_{\substack{1\le j\le n\\(j,n)=1}}e\left(\frac{ mj}{n}\right)=\sum_{\substack{d|n\\d|m}}d\mu\left(\frac{n}{d}\right)\numberthis\label{rsum}\] now known as the Ramanujan sum, where $m$ and $n$ are positive integers, $e(x)=e^{2\pi ix}$, and $\mu(n)$ is the Mobius function. 
Understanding these sums and their distribution is an important topic of study in number
theory, with profound connections to problems in arithmetic such as in the proof of Vinogradov's theorem \cite[Chapter 8]{nathanson1996additive}, Waring type formulas \cite{konvalina1996generalization}, distribution of rational numbers in short intervals \cite{jutila2007distribution}, equipartition modulo odd integers \cite{balandraud2007application}, large sieve inequality \cite{ramare2007eigenvalues}, as well as other areas of mathematics.

Ramanujan sums have been generalized by many mathematicians in several contexts. Some examples include Cohen-Ramanujan sums \cite{sugunamma1960eckford},
Anderson-Apostol sums \cite{anderson1953evaluation}, polynomial Ramanujan sums introduced by Carlitz \cite{carlitz1947singular} and later generalized by Cohen \cite{cohen1949extension}, and Ramanujan sums in the more general context of arithmetic subgroups (see \cite{grytczuk1992ramanujan,knopfmacher2015abstract}).

The main purpose of this note is to study the distribution of Ramanujan sums defined over number fields by examining moments of its mean values. Let $\mathbb{K}$ be a number field and $\mathcal{J}$ and $\mathcal{I}$ be non-zero integral ideal in its number ring $\mathcal{O}_{\mathbb{K}}$, then Ramanujan sums over $\mathbb{K}$ are defined as
        \[C_{\mathcal{J}}({\mathcal{I}}):=\sum_{\substack{\mathcal{I}_1|\mathcal{J}\\\mathcal{I}_1|\mathcal{I}}}\mathcal{N}(\mathcal{I}_1)\mu(\frac{\mathcal{J}}{\mathcal{I}_1}).\numberthis\label{one}\]
         Here, $\mathcal{N}(\mathcal{I}_1)$ is the norm of $\mathcal{I}_1$ and $\mu(\mathcal{I})$ is the generalization of classical Mobius function such that \[\mu(\mathcal{I})=\left\{\begin{array}{ll}
       (-1)^r  & \mbox{if } \mathcal{I}\ \text{is a product of r distinct prime ideals},  \\
       0  &  \mbox{if } \text{there exists a prime ideal}\ \mathcal{P} \ \text{of}\ \mathcal{O}_{\mathbb{K}}\  \text{such that} \ \mathcal{P}^2|\mathcal{I}.
    \end{array}\right.\]
    Note that for $\mathbb{K}=\mathbb{Q}$, it is the usual Ramanujan sum $c_n(m)$ in \eqref{rsum}.
    The question on the average order over both variables $n$ and $m$ of $c_n(m)$ was first considered by Chan and Kumchev \cite{MR2869206} motivated by applications to problems on Diophantine approximations of reals by sums of rational numbers. In \cite{MR2869206}, using both elementary and analytic techniques, they  
find asymptotic formulas for
\[\sum_{m\le y}\left(\sum_{n\le x}c_n(m)\right)^k\] for $k=1,2$. 
Robles and Roy \cite{MR3600410} study first, second, and higher moments of averages of Cohen-Ramanujan sums. Although their result for higher moments  ($k\ge 3$) (Proposition 1.1) is incorrect, and the problem of computing asymptotic formulas for higher moments even for the usual Ramanujan sums \eqref{rsum} remains open. 
With regard to the number field analogue of Ramanujan sums, Nowak \cite{Nowak12} showed that if $\mathbb{K}$ is a fixed quadratic number field, and $y>x^{\delta}$ where $\delta>\frac{1973}{820}=2.40609\cdots$, then 
\[\sum_{0<\mathcal{N}(\mathcal{I})\le y}\sum_{0<\mathcal{N}(\mathcal{J})\le x}C_{\mathcal{J}}({\mathcal{I}})\sim \rho_{\mathbb{K}} y,\numberthis\label{nowak}\]
More precisely, for $y>x$ and arbitrary $\epsilon>0$,
 \begin{align*}
            \sum_{0<\mathcal{N}(\mathcal{I})\le y}\sum_{0<\mathcal{N}(\mathcal{J})\le x}C_{\mathcal{J}}({\mathcal{I}})=&\rho_{\mathbb{K}} y+\BigO{x^{\frac{1973}{1358}}y^{\frac{269}{679}+\epsilon}}+\BigO{x^{\frac{1234823}{737394}}y^{\frac{205}{679}+\epsilon}}+\BigO{x^{\frac{23917}{21728}}y^{\frac{8675}{16296}+\epsilon}}\\&+ \BigO{x^2y^{\epsilon}},\end{align*}
where 
\[\rho_{\mathbb{K}}=\lim_{t\rightarrow\infty}\frac{1}{t}\#\{\text{integral ideals }  \mathcal{I} \text{ in } O_{\mathbb{K}}: 0< N(\mathcal{I})\le t\}
\numberthis\label{residue}\]
In this note, we estimate the sum
\[\sum_{0<\mathcal{N}(\mathcal{J})\le x}C_{\mathcal{J}}({\mathcal{I}})\] in average over ideals $\mathcal{I}$ such that $\mathcal{N}(\mathcal{I})\in\{1, \cdots, y\}$ via the second moment. This generalizes the results in \cite{MR3600410} for second moments of mean value of Cohen-Ramanujan sums. For a quadratic number field $\mathbb{K}$, we derive the following asymptotic formula.
\begin{thm}\label{theorem2}
Let $\mathbb{K}$ be a quadratic number field, then for any $\epsilon>0$, and $x\le y <x^2$ 
   \begin{align*}
        \sum_{0<\mathcal{N}(\mathcal{I})\le y} \left(\sum_{0<\mathcal{N}(\mathcal{J})\le x}C_{\mathcal{J}}({\mathcal{I}})\right)^2&= \frac{\rho_{\mathbb{K}}^2}{2\zeta_{\mathbb{K}}(2)}x^2y-\frac{\rho_{\mathbb{K}}^2\zeta_{\mathbb{K}}(0)}{4\zeta_{\mathbb{K}}(2)^2 }x^4+\BigO{yx^{2-\epsilon}\log^7 x+x^2y^{5/6}\log^{24}x},
   \end{align*}
   and for $x^2\le y <x^3$
    \begin{align*}
        \sum_{0<\mathcal{N}(\mathcal{I})\le y} \left(\sum_{0<\mathcal{N}(\mathcal{J})\le x}C_{\mathcal{J}}({\mathcal{I}})\right)^2&= \frac{\rho_{\mathbb{K}}^2}{2\zeta_{\mathbb{K}}(2)}x^2y+\BigO{yx^{2-\epsilon}\log^7 x+x^2y^{5/6}\log^{24}x}.
   \end{align*}
\end{thm}
 With regard to other degree two extensions, for example for the field of Gaussian integers, Nowak \cite{nowak13} proved \eqref{nowak} with uniform error terms  with $\delta>\frac{29}{12}=2.416\cdots$. His result was later improved in \cite{zhai2021average} for $\delta>2.3235\cdots$. 
 
For the cubic case, a result on the first moment is derived in \cite{ma2021average} 
where the authors obtain an asymptotic formula \eqref{nowak} with condition $y>x^{11/4}$. 
 \begin{align*}
            \sum_{0<\mathcal{N}(\mathcal{I})\le y}\sum_{0<\mathcal{N}(\mathcal{J})\le x}C_{\mathcal{J}}({\mathcal{I}})=&\rho_{\mathbb{K}} y+\BigO{x^{\frac{8}{5}}y^{\frac{2}{5}+\epsilon}+x^{\frac{11}{8}}y^{\frac{1}{2}+\epsilon}}.\end{align*}
We obtain estimates on the second moment for a cubic number field in the following theorem.
\begin{thm}\label{theorem4}
Let $\mathbb{K}$ be a cubic number field, then for any  $\epsilon>0$, and $x\le y <x^2$ 
   \begin{align*}
        \sum_{0<\mathcal{N}(\mathcal{I})\le y} \left(\sum_{0<\mathcal{N}(\mathcal{J})\le x}C_{\mathcal{J}}({\mathcal{I}})\right)^2&= \frac{\rho_{\mathbb{K}}^2}{2\zeta_{\mathbb{K}}(2)}yx^2-\frac{\rho_{\mathbb{K}}^2\zeta_{\mathbb{K}}(0)}{4\zeta_{\mathbb{K}}(2)^2 }x^4+\BigO{yx^{2-\epsilon}\log^{10} x +x^2y^{13/16}\log^{31}x},
   \end{align*}
   and for $x^2\le y <x^3$
    \begin{align*}
        \sum_{0<\mathcal{N}(\mathcal{I})\le y} \left(\sum_{0<\mathcal{N}(\mathcal{J})\le x}C_{\mathcal{J}}({\mathcal{I}})\right)^2&= \frac{\rho_{\mathbb{K}}^2}{2\zeta_{\mathbb{K}}(2)}yx^2+\BigO{yx^{2-\epsilon}\log^{10} x+x^2y^{13/16}\log^{31}x}.
   \end{align*}
   \end{thm}
  For mean values of Ramanujan sums over general number fields, the only known result is due to Fujisawa \cite{MR3332952} who proved that if $\mathbb{K}$ is any number field, then for some $c>0$, and for any $\delta>\frac{2-\alpha}{1-\alpha}$ where $\alpha\in [0,1)$, with condition $y\ll x^{\delta}$, \[\sum_{0<\mathcal{N}(\mathcal{I})\le y}\sum_{0<\mathcal{N}(\mathcal{J})\le x}C_{\mathcal{J}}({\mathcal{I}})= \rho_{\mathbb{K}}y+o(y).\numberthis\label{fujisawa}\]
  His result is a consequence of a more general theorem on moments of Ramanujan sums over Dedekind domains \cite[Theorem 1] {MR3332952}. In the next theorem, we derive asymptotic results for the second moment of Ramanujan sums over Prüfer domains
   \begin{defn}
   An integral domain $R$ is called a Pr\"ufer domain if every finitely generated non-zero ideal of $R$
is invertible.
\end{defn}
For our computations of the second moment, we use the following ideal property of Pr\"ufer domains:
If $\mathcal{I}$, and $\mathcal{J}$ are two ideals of a Prüfer domain, then 
   \[(\mathcal{I}+\mathcal{J})(\mathcal{I}\cap\mathcal{J})=\mathcal{I}\mathcal{J}. \numberthis\label{property1}\]
  Some examples of a Pr\"ufer domain consist of the ring of algebraic integers, the ring of entire functions in $\mathbb{C}$. For more on multiplicative ideal theory and Pr\"ufer domains, see \cite[Chapter 4]{gilmer1992multiplicative}.
   
\begin{thm}\label{theorem5}
 Let $\mathbb{K}$ be a number field such that its ring of integers $\mathcal{O}_{\mathbb{K}}$ is a  Pr\"ufer domain. if \[\sum_{1\le \mathcal{N}(\mathcal{I})\le y}1=\rho_{\mathbb{K}} y+\BigO{y^{\alpha}},\]then for $x^{\lambda}<y$  for some $\lambda>1$, we have
         \[ \sum_{0<\mathcal{N}(\mathcal{I})\le y} \left(\sum_{0<\mathcal{N}(\mathcal{J})\le x}C_{\mathcal{J}}({\mathcal{I}})\right)^2=\dfrac{\rho_{\mathbb{K}}^2}{2\zeta_{\mathbb{K}}(2)}x^2y+\BigO{xy\log x+ x^{3-\alpha}y^{\alpha}}.\]
\end{thm}
The value of $\alpha$ was estimated by Landau \cite{landau1949einfuhrung} to be $(n-1)/(n+1)$, where $n$ is the degree of $\mathbb{K}$ over $\mathbb{Q}$. This was later improved by Nowak \cite{Nowak12} and M\"uller \cite{muller1988distribution} for the case $n=2$ and $n=3$, respectively.
\begin{rem} \label{remarkone}
Theorem \ref{theorem5} holds for any number field whose corresponding ring of integers satisfies property \eqref{property1}. For the ring of integers $\mathbb{Z}$, \eqref{property1} reduces to the fact that the gcd times lcm of any two integers is equal to the product of the integers. This property is not valid for more than two integers, complicating the computations for higher moments $(k\ge 3)$.
\end{rem}
\begin{rem}\label{remarktwo}
 Theorems \ref{theorem2} and \ref{theorem4} are special cases of Theorem \ref{theorem5} with additional main terms in certain ranges of $y$, as the ring of integers for both quadratic and cubic number field is a  Prüfer domain. The constants in Theorems \ref{theorem2} and \ref{theorem4} depend on the discriminant of the corresponding number fields.
\end{rem}
\subsection{Organization}
This article is organized as follows. Section \ref{sec1} covers preliminary results required to prove Theorems \ref{theorem2}, \ref{theorem4}, and \ref{theorem5}. Section \ref{sec2} contains a key result involving average of product of divisor functions over number fields. Section \ref{sec3} contains proofs of Theorems \ref{theorem2} and \ref{theorem4} invoking the key estimate proved in Section \ref{sec2}. Finally, Section \ref{sec4} contains proof of Theorem \ref{theorem5}.
\subsection{Notations} Throughout this note, we use $\mathbb{Z}$, $\mathbb{Q}$, and $\mathbb{K}$ to denote the set of integers, set of rational numbers, and a number field, respectively. We denote complex numbers $z=\sigma+it$, $z_1=a_1+ib_1$, and $z_2=a_2+ib_2$.  We use $\phi$
to denote the Euler totient function, $\mu$ the Mobius function, $\zeta_{\mathbb{K}}(s)$ the Dedekind zeta function corresponding to a number field $\mathbb{K},$ and $\zeta(s)$ the Riemann zeta function. We use the Vinogradov $\ll$ asymptotic notation, and the big oh O(·) and o(·)
asymptotic notation. Dependence on a parameter will be denoted by a subscript.
\subsection{Acknowledgements}
The first author is grateful for the support from the Science and Engineering Research Board, Department of Science and Technology, Government of India under grant SB/S2/RJN-053/2018.
\section{Preliminaries} \label{sec1}
In this section, we state and prove some results related to the Dirichlet series of functions appearing in the proofs of Theorems \ref{theorem2} and \ref{theorem4}. We start by recalling the Dirichlet series of $C_{\mathcal{J}}({\mathcal{I}})$.
\begin{lem}\label{lemmaone}
For a number field $\mathbb{K}$ and for $\Re(s)>1$, one has\[\sum_{\mathcal{J}\subseteq \mathcal{O}_{\mathbb{K}}}\frac{C_{\mathcal{J}}({\mathcal{I}})}{\mathcal{N}(\mathcal{J})^s}=\frac{\sigma_{\mathbb{K},(1-s)}(\mathcal{I})}{\zeta_{\mathbb{K}}(s)},\numberthis\label{two}\]
where $\sigma_{\mathbb{K},(1-s)}(\mathcal{I})=\sum_{\mathcal{I}_1|\mathcal{I}}\mathcal{N}(\mathcal{I}_1)^{1-s}.$
\end{lem}
\begin{proof}
From the definition of $C_{\mathcal{J}}({\mathcal{I}})$ in \eqref{one}, we have 
\begin{align*}
 \sum_{\mathcal{J}\subseteq \mathcal{O}_{\mathbb{K}}}\frac{C_{\mathcal{J}}({\mathcal{I}})}{\mathcal{N}(\mathcal{J})^s}&=  \sum_{\mathcal{J}\subseteq \mathcal{O}_{\mathbb{K}}}\frac{1}{\mathcal{N}(\mathcal{J})^s}\sum_{\substack{\mathcal{I}_1|\mathcal{J}\\\mathcal{I}_1|\mathcal{I}}}\mathcal{N}(\mathcal{I}_1)\mu(\frac{\mathcal{J}}{\mathcal{I}_1})=\sum_{\mathcal{I}_1|\mathcal{I}}\dfrac{1}{\mathcal{N}(\mathcal{I}_1)^{s-1}}\sum_{\mathcal{I}_2\subseteq \mathcal{O}_{\mathbb{K}}}\frac{\mu(\mathcal{I}_2)}{\mathcal{N}(\mathcal{I}_2)^s}\\
 &=\frac{\sigma_{\mathbb{K},(1-s)}(\mathcal{I})}{\zeta_{\mathbb{K}}(s)}.
\end{align*}
The above series is absolutely convergent for $\Re(s)>1$.
\end{proof} 
\begin{lem}\label{lemmatwo}
For $z\in \mathbb{C}$, and for a number field $\mathbb{K}$,
\[\sum_{\mathcal{I}\subseteq \mathcal{O}_{\mathbb{K}}}\frac{\sigma_{\mathbb{K},z}(\mathcal{I})}{\mathcal{N}(\mathcal{I})^s}=\zeta_{\mathbb{K}}(s)\zeta_{\mathbb{K}}(s-z),\numberthis\label{three}\]
for $\Re(s)>\max(1+\Re(z),1).$
\end{lem}
\begin{proof}
Using the definition of $\sigma_{\mathbb{K},z}(\mathcal{I})$, we have
\begin{align*}
  \sum_{\mathcal{I}\subseteq \mathcal{O}_{\mathbb{K}}}\frac{\sigma_{\mathbb{K},z}(\mathcal{I})}{\mathcal{N}(\mathcal{I})^s}&=  \sum_{\mathcal{I}\subseteq \mathcal{O}_{\mathbb{K}}}\frac{1}{\mathcal{N}(\mathcal{I})^s}\sum_{\mathcal{I}_1|\mathcal{I}}\mathcal{N}(\mathcal{I}_1)^{z}\\
  &= \zeta_{\mathbb{K}}(s)\zeta_{\mathbb{K}}(s-z),
\end{align*}
and it is absolutely convergent in  $\Re(s)>\max(1+\Re(z),1).$
\end{proof}

\begin{lem}\label{lemmafive}
For $\Re(s)>\max(1,1+\Re(z_1),1+\Re(z_2),1+\Re(z_1+z_2))$, we have\[\sum_{\mathcal{I}\subseteq \mathcal{O}_{\mathbb{K}}}\frac{\sigma_{\mathbb{K},z_1}(\mathcal{I})\sigma_{\mathbb{K},z_2}(\mathcal{I})}{\mathcal{N}(\mathcal{I})^s}=\frac{\zeta_{\mathbb{K}}(s)\zeta_{\mathbb{K}}(s-z_1)\zeta_{\mathbb{K}}(s-z_2)\zeta_{\mathbb{K}}(s-z_1-z_2)}{\zeta_{\mathbb{K}}(2s-z_1-z_2)}.\numberthis\label{sigm,adirichlet}\]
\end{lem}
\begin{proof}
Fix a prime ideal $\mathcal{P}$, then for a positive integer $k$, \[\sigma_{\mathbb{K},z}(\mathcal{P}^k)=\frac{\mathcal{N}\mathcal({P})^{z(k+1)}-1}{\mathcal{N}\mathcal({P})^{z}-1}.\] Both functions $\sigma_{\mathbb{K},z}(\mathcal{I})$, and $\sigma_{\mathbb{K},z_1}(\mathcal{I})\sigma_{\mathbb{K},z_2}(\mathcal{I})$ are multiplicative, and hence the infinite series has an Euler product representation given by 
\begin{align*}
 \sum_{\mathcal{I}\subseteq \mathcal{O}_{\mathbb{K}}}\frac{\sigma_{\mathbb{K},z_1}(\mathcal{I})\sigma_{\mathbb{K},z_2}(\mathcal{I})}{\mathcal{N}(\mathcal{I})^s}&= \prod_{\mathcal{P}\subseteq \mathcal{O}_{\mathbb{K} }}\left(1+\sum_{k=1}^{\infty} \frac{\sigma_{\mathbb{K},z_1}(\mathcal{P}^k)\sigma_{\mathbb{K},z_2}(\mathcal{P}^k)}{\mathcal{N}(\mathcal{P})^{ks}}\right) \\
 & = \prod_{\mathcal{P}\subseteq \mathcal{O}_{\mathbb{K} }}\left(1+\sum_{k=1}^{\infty}\frac{(\mathcal{N}\mathcal({P})^{z_1(k+1)}-1)(\mathcal{N}\mathcal({P})^{z_2(k+1)}-1)}{\mathcal{N}(\mathcal{P})^{ks}(\mathcal{N}\mathcal({P})^{z_1}-1)(\mathcal{N}\mathcal({P})^{z_2}-1)}\right).
\end{align*}
Let $\mathcal{N}(\mathcal{P})^{-s}=x$, $\mathcal{N}(\mathcal{P})^{z_1}=y$, and $\mathcal{N}(\mathcal{P})^{z_2}=z$, then \begin{align*}
    \sum_{\mathcal{I}\subseteq \mathcal{O}_{\mathbb{K}}}\frac{\sigma_{\mathbb{K},z_1}(\mathcal{I})\sigma_{\mathbb{K},z_2}(\mathcal{I})}{\mathcal{N}(\mathcal{I})^s}&= \prod_{\mathcal{P}\subseteq \mathcal{O}_{\mathbb{K} }}\left(\frac{1}{(y-1)(z-1)}\sum_{k=0}^{\infty}x^k(y^{k+1}-1)(z^{k+1}-1)\right)\\
    &= \prod_{\mathcal{P}\subseteq \mathcal{O}_{\mathbb{K} }}\left(\frac{1}{(y-1)(z-1)}\left\{\frac{yz}{1-xyz}-\frac{z}{1-xz}-\frac{y}{1-xy}+\frac{1}{1-x}\right\}\right)\\
    &=\prod_{\mathcal{P}\subseteq \mathcal{O}_{\mathbb{K} }}\frac{1-x^2yz}{(1-x)(1-xy)(1-xz)(1-xyz)}.
\end{align*}
On substituting the values of x, y , and z in the above equation, we obtain Lemma \ref{lemmafive}.
\end{proof}
Next, we cite two lemmas which will be useful in the next section. The first one is a Brun-Titchmarsh theorem proved by Shiu \cite{shiu1980brun}. We will use it to estimate the partial sum
\[\sum_{\substack{\mathcal{I}\subseteq \mathcal{O}_{\mathbb{K}}\\\mathcal{N}(\mathcal{I})=n}}\sigma_{\mathbb{K},z_1}(\mathcal{I})\sigma_{\mathbb{K},z_2}(\mathcal{I}).\]
In \cite{shiu1980brun}, the author derives the theorem for a larger class $M$ of arithmetic function $f$ which are non-negative and multiplicative, and which satisfy the following conditions:
\begin{enumerate}
    \item For a prime $p$, and integer $l\ge 1$, there exists a positive constant $C_1$ such that \[f(p^l)\le C_1^l,\]
    \item For every $\epsilon>0$, and for $n\ge 1$, there exists a positive constant $C_2=C_2(\epsilon)$ such that \[f(n)\le C_2n^{\epsilon}.\]
\end{enumerate}
\begin{lem}\cite[Theorem 1]{shiu1980brun} \label{averagelemma}
Let $f\in M$, $0<\alpha, \beta<1/2$, and $a,k$ be integers. If $0<a<k$, and  $(a,k)=1$, then as $x\to \infty$
\[\sum_{\substack{x-y<n\le x\\n\equiv a\mod{q}}} f(n)\ll \frac{y}{\phi(q)\log x}\exp{\left(\sum_{\substack{p\le x\\p\not|q}}\frac{f(p)}{p}\right)},\] uniformly in $a,q$, and $y$ provided that $q\le y^{1-\alpha}$, and  $x^{\beta}<y\le x$.
\end{lem}
The second lemma is a Perron-type formula for a sequence of complex numbers.
\begin{lem}\cite[Lemma 2.8]{MR3600410}\label{parronlemma}
Let $0<\lambda_1<\lambda_2<\cdots<\lambda_n\to \infty$ be any sequence of real numbers, and let $\{a_n\}$ be a sequence of complex numbers. Let the Dirichlet series $g(s):=\sum_{n=1}^{\infty}a_n\lambda_{n}^{-s}$ be absolutely convergent for $\sigma_a$. If $\sigma_0>\max(0,\sigma_a)$ and $x>0$, then \[\sum_{\lambda_n\le x}a_n=\frac{1}{2\pi i}\int_{\sigma_0-iT}^{\sigma_0+iT}g(s)\frac{x^s}{s}ds+R,\] where \[R\ll \sum_{\substack{x/2<\lambda_n<2x\\n\ne x}}|a_n|\min\left(1,\frac{x}{T|x-\lambda_n|}\right)+\frac{4^{\sigma_0}+x^{\sigma_0}}{T}\sum_{n=1}^{\infty}\frac{|a_n|}{\lambda_n^{\sigma_0}}.\] 
\end{lem}
For a quadratic number field $\mathbb{K}$ with discriminant q,  \[\zeta_{\mathbb{K}}(s)=\zeta(s)L(s,\chi_q),\numberthis\label{qseriesrepresentation}\] where $\zeta(s)$ is the Riemann zeta function, and $L(s,\chi_q)$ is the ordinary Dirichlet  L-series corresponding to  $\chi_q$ and $\chi_q$ is the Kronecker symbol of q.   We use bounds of $\zeta(s)$ \cite[Page 47]{MR882550} and $L(s,\chi)$ \cite{MR551704}, and derive the following bounds for $\zeta_{\mathbb{K}}(s)$: 
\begin{displaymath}
 \zeta_{\mathbb{K}}(\sigma+it) \ll_q \left\{
   \begin{array}{lr}
    {t^{1-2\sigma} \log^2 t},& \  -1\le \sigma \le 0,\\
     {t^{\frac{7- 10\sigma}{6}} \log^4 t}, &\  0\le \sigma \le 1/2,\\
      {t^{\frac{2-2\sigma}{3}} \log^4 t}, &\  1/2\le \sigma \le 1,\\
      { \log^2 t},  &\ 1\le\sigma \le 2,\\
       {1} , &\ \sigma \ge 2,
     \end{array}\numberthis\label{bound}
    \right.
\end{displaymath}
The next lemma gives an expression for the Dedekind zeta function for a cubic number field with discriminant $D=dq^2$ ($d$ squarefree).
\begin{lem}\label{cubic}\cite[Lemma 1]{muller1988distribution}
Let $\mathbb{K}$ be a cubic number field and $D=dq^2$ ($d$ squarefree) its discriminant; then 
\begin{enumerate}
    \item $\mathbb{K}$ is normal extension if and only if $D=q^2$. In this case \[\zeta_{\mathbb{K}}(s)=\zeta(s)L(s,\chi_1)\overline{L(s,\chi_1)},\numberthis\label{cseriesrepresentation1}\]where $\zeta(s)$ is the Riemann zeta function and $L(s,\chi_1)$ is the ordinary Dirichlet series corresponding to the primitive character $\chi_1$ modulo q.
    \item If $\mathbb{K}$ is not a normal extension, then $d\ne 1$, and 
    \[\zeta_{\mathbb{K}}(s)=\zeta(s)L(s,\chi_2),\numberthis\label{cseriesrepresentation2}\] where $L(s,\chi_2)$ is the Dirichlet L-series over the quadratic number field $\mathbb{Q}(\sqrt{d})$: \[L(s,\chi_2)=\sum_{\mathcal{I}}\chi_{2}(\mathcal{I})\mathcal{N}(\mathcal{I})^{-s}.\]Here summation is taken over all ideals $\mathcal{I}\ne 0$ in  $\mathbb{Q}(\sqrt{d})$. 
\end{enumerate}
\end{lem}
Using the above lemma, we arrive at the following bounds in the cubic case.
\begin{displaymath}
 \zeta_{\mathbb{K}}(\sigma+it) \ll_q \left\{
   \begin{array}{lr}
    {t^{3(1/2-\sigma)} \log^3 t},& \  -1\le \sigma \le 0,\\
    {t^{\frac{11- 17\sigma}{6}} \log^7 t}, &\  0\le \sigma \le 1/2,\\
    {t^{\frac{5(1-\sigma)}{6}} \log^7 t}, &\  1/2\le \sigma \le 1,\\
      { \log^3 t},  &\ 1\le\sigma \le 2,\\
       {1} , &\ \sigma \ge 2.
     \end{array}\numberthis\label{cbound}
    \right.
\end{displaymath}
\section{A key estimate} \label{sec2}
In this section, we shall compute the average of product of divisor functions in a number field, analogous to \cite[Theorem 1.5]{MR3600410} for divisor functions over rationals. 
\begin{lem}\label{lemmasix}
Let $\mathbb{K}$ be a number field. Then,
\[\sum_{\substack{0<\mathcal{N}(\mathcal{I})\le x\\\mathcal{I}\subseteq \mathcal{O}_{\mathbb{K}}}}\sigma_{\mathbb{K},z_1}(\mathcal{I})\sigma_{\mathbb{K},z_2}(\mathcal{I})=\mathrel{R}_{\mathbb{K}}+E_{\mathbb{K}}.\]
For a \textbf{quadratic number field $\mathbb{K}$}, and for $-1/3<a_1<0$, $-2/9<a_2<0$, and $-2/9<a_1+a_2<0$, 
\begin{align*}
  \mathrel{R}_{\mathbb{K}}= &\rho_{\mathbb{K}} \frac{\zeta_{\mathbb{K}}(1-z_1)\zeta_{\mathbb{K}}(1-z_2)\zeta_{\mathbb{K}}(1-z_1-z_2)}{\zeta_{\mathbb{K}}(2-z_1-z_2)}{x}+\rho_{\mathbb{K}}\frac{\zeta_{\mathbb{K}}(1+z_1)\zeta_{\mathbb{K}}(1+z_1-z_2)\zeta_{\mathbb{K}}(1-z_2)}{\zeta_{\mathbb{K}}(2+z_1-z_2)}\frac{x^{1+z_1}}{1+z_1}\\&+\rho_{\mathbb{K}}\frac{\zeta_{\mathbb{K}}(1+z_2)\zeta_{\mathbb{K}}(1+z_2-z_1)\zeta_{\mathbb{K}}(1-z_1)}{\zeta_{\mathbb{K}}(2-z_1+z_2)}\frac{x^{1+z_2}}{1+z_2}\\&+\rho_{\mathbb{K}} \frac{\zeta_{\mathbb{K}}(1+z_1+z_2)\zeta_{\mathbb{K}}(1+z_2)\zeta_{\mathbb{K}}(1+z_1)}{\zeta_{\mathbb{K}}(2+z_1+z_2)}\frac{x^{1+z_1+z_2}}{1+z_1+z_2},\end{align*}and
  \[E_{\mathbb{K}}=\BigOq{x^{\frac{10+6a_1+3a_2}{12}}\log^{18}x}.\]
  For \textbf{a cubic number field $\mathbb{K}$}, and for
  $-3/16<a_1<0$, $-3/16<a_2<0$, and $-3/16<a_1+a_2<0$,
  $\mathrel{R}_{\mathbb{K}}$ is same as above, and
  \[E_{\mathbb{K}}=\BigOq{x^{\frac{13}{16}}\log^{31}x}.\]
\end{lem}
\begin{proof}
For any number field $\mathbb{K}$, one has \[\sum_{\mathcal{I}\subseteq \mathcal{O}_{\mathbb{K}}}\frac{\sigma_{\mathbb{K},z_1}(\mathcal{I})\sigma_{\mathbb{K},z_2}(\mathcal{I})}{\mathcal{N}(\mathcal{I})^s}=\sum_{n=1}^{\infty}\frac{1}{n^s}\sum_{\substack{\mathcal{I}\subseteq \mathcal{O}_{\mathbb{K}}\\\mathcal{N}(\mathcal{I})=n}}\sigma_{\mathbb{K},z_1}(\mathcal{I})\sigma_{\mathbb{K},z_2}(\mathcal{I}).\] 
Define $A(n,z_1,z_2):=\sum_{\substack{\mathcal{I}\subseteq \mathcal{O}_{\mathbb{K}}\\\mathcal{N}(\mathcal{I})=n}}\sigma_{\mathbb{K},z_1}(\mathcal{I})\sigma_{\mathbb{K},z_2}(\mathcal{I})$ and, \[f(z_1,z_2,s):=\dfrac{\zeta_{\mathbb{K}}(s)\zeta_{\mathbb{K}}(s-z_1)\zeta_{\mathbb{K}}(s-z_2)\zeta_{\mathbb{K}}(s-z_1-z_2)}{\zeta_{\mathbb{K}}(2s-z_1-z_2)}.\]
Let $\Re(z_1)=a_1$ and $\Re(z_2)=a_2$ be such that $a_1,a_2<0$, and $a_1+a_2>-1$. Consider $\alpha=1+\dfrac{1}{\log x}$. Using Lemma \ref{parronlemma}, we have \begin{align*}
    \sum_{\substack{0<\mathcal{N}(\mathcal{I})\le x\\\mathcal{I}\subseteq \mathcal{O}_{\mathbb{K}}}}\sigma_{\mathbb{K},z_1}(\mathcal{I})\sigma_{\mathbb{K},z_2}(\mathcal{I})=\sum_{n\le x}A(n,z_1,z_2)=&\frac{1}{2\pi i}\int_{\alpha-iT}^{\alpha+iT}f(z_1,z_2,s)\frac{x^s}{s}ds\\&+R(x;z_1,z_2),\numberthis\label{parron4}\end{align*} where \[R(x;z_1,z_2)\ll \sum_{x/2<n<2x}|A(n,z_1,z_2)|\min\left(1,\frac{x}{T|x-n|}\right)+\frac{x^{\alpha}}{T}\sum_{n=1}^{\infty}\frac{|A(n,z_1,z_2)|}{n^{\alpha}}.\numberthis\label{errordoublesum}\]
 For $\lambda=(1+a_1+a_2)/2$, we solve integral in \eqref{parron4} by modifying the line integral into a rectangular path $\mathrel{C}$ with vertices $\alpha\pm iT$, and $\lambda\pm iT$. The poles of the integrand inside the contour $\mathrel{C}$ are $s_1=1$, $s_2=z_1+1$, $s_3=z_2+1$, and $s_4=z_1+z_2+1$. Hence, by Cauchy's residue theorem we have   
\begin{align*}
    \frac{1}{2\pi i}\int_{\mathrel{C}}f(z_1,z_2,s)\frac{x^s}{s}ds&=\rho_{\mathbb{K}} \frac{\zeta_{\mathbb{K}}(1-z_1)\zeta_{\mathbb{K}}(1-z_2)\zeta_{\mathbb{K}}(1-z_1-z_2)}{\zeta_{\mathbb{K}}(2-z_1-z_2)}{x}\\&+\rho_{\mathbb{K}}\frac{\zeta_{\mathbb{K}}(1+z_1)\zeta_{\mathbb{K}}(1+z_1-z_2)\zeta_{\mathbb{K}}(1-z_2)}{\zeta_{\mathbb{K}}(2+z_1-z_2)}\frac{x^{1+z_1}}{1+z_1}\\&+\rho_{\mathbb{K}}\frac{\zeta_{\mathbb{K}}(1+z_2)\zeta_{\mathbb{K}}(1+z_2-z_1)\zeta_{\mathbb{K}}(1-z_1)}{\zeta_{\mathbb{K}}(2-z_1+z_2)}\frac{x^{1+z_2}}{1+z_2}\\&+\rho_{\mathbb{K}} \frac{\zeta_{\mathbb{K}}(1+z_1+z_2)\zeta_{\mathbb{K}}(1+z_2)\zeta_{\mathbb{K}}(1+z_1)}{\zeta_{\mathbb{K}}(2+z_1+z_2)}\frac{x^{1+z_1+z_2}}{1+z_1+z_2}.\numberthis\label{residue}\end{align*}
 This implies  
 \begin{align*}
     \frac{1}{2\pi i}\int_{\alpha-iT}^{\alpha+iT}f(z_1,z_2,s)\frac{x^s}{s}ds&= \mathrel{R_0}+\sum_{i=1}^{3}J_i,\numberthis\label{finalintegral}
 \end{align*}
 where $\mathrel{R_0}$ is equal to the right side of \eqref{residue}, and $J_i$'s are the line integrals along the lines $[\alpha+iT,\lambda+iT]$, $[\lambda-iT,\lambda+iT]$ and $[\alpha-iT,\lambda-iT]$  respectively. By Holder's inequality 
 \begin{align*}
   \left(\int_{\lambda}^{\alpha}\int_{T_0/2}^{T_0}f(z_1,z_2,\sigma+it)\frac{x^{\sigma+it}}{\sigma+it}d\sigma dt \right)^4&\ll  \int_{\lambda}^{\alpha}\int_{T_0/2}^{T_0}\frac{|\zeta_{\mathbb{K}}(\sigma+it)|^4x^{\sigma}}{|\zeta_{\mathbb{K}}(2(\sigma+it)-z_1-z_2)(\sigma+it)|}d\sigma dt\\&\times\int_{\lambda}^{\alpha}\int_{T_0/2}^{T_0}\frac{|\zeta_{\mathbb{K}}(\sigma+it-z_1)|^4x^{\sigma}}{|\zeta_{\mathbb{K}}(2(\sigma+it)-z_1-z_2)(\sigma+it)|}d\sigma dt\\ &\times \int_{\lambda}^{\alpha}\int_{T_0/2}^{T_0}\frac{|\zeta_{\mathbb{K}}(\sigma+it-z_2)|^4x^{\sigma}}{|\zeta_{\mathbb{K}}(2(\sigma+it)-z_1-z_2)(\sigma+it)|}d\sigma dt \\& \times \int_{\lambda}^{\alpha}\int_{T_0/2}^{T_0}\frac{|\zeta_{\mathbb{K}}(\sigma+it-z_1-z_2)|^4x^{\sigma}}{|\zeta_{\mathbb{K}}(2(\sigma+it)-z_1-z_2)(\sigma+it)|}d\sigma dt.\numberthis\label{holder}
 \end{align*}
 We estimate the above integrals, and the remainder $R(x;z_1,z_2)$ in \eqref{errordoublesum} separately for the quadratic and cubic number fields. 
 \subsection{Quadratic Number Field}
For $\mathbb{K}$ a quadratic number field, the Dirichlet series for $A(n,z_1,z_2)$ can be expressed in terms of $\zeta(s)$ and $L(s,\chi)$ using \eqref{sigm,adirichlet} and \eqref{qseriesrepresentation}. Consequently, $A(n,z_1,z_2)$ is written as a Dirichlet convolution of coefficients of its Dirichlet series. This exercise yields the following elementary but essential bound. \begin{align*}
      |A(n,z_1,z_2)|&\le \left |n^{a_1+a_2}\sum_{\mathcal{N}(\mathcal{I})=n}\left(\sum_{\mathcal{I}_1|\mathcal{I}}1\right)^2\right|\\
      &\le \sum_{d|n}\left\{\sum_{d'|d}\left(\sigma_0(d')\sum_{d_1|d/d'}\sigma_0(d_1)\sigma_0(\frac{d}{d'd_1})\right)\right\}\sigma_0(n/d).
\numberthis\label{quadraticinequality}  \end{align*}
   Hence, \[|A(p,z_1,z_2)|\le 9.\]
  Taking $T=x^c$ where $c$ is a fixed real number, and dividing the interval $x/2<n<2x$ according to $\min\left(1,\dfrac{x}{T|x-n|}\right)$, we arrive at
  \begin{align*}
       \sum_{|x-n|<x^{1-c}}|A(n,z_1,z_2)|\min\left(1,\frac{x}{T|x-n|}\right)&=\sum_{|x-n|<x^{1-c}}|A(n,z_1,z_2)|\ll \frac{x}{T}\log^{8}T,\numberthis\label{errordouble}.
   \end{align*}
   The last estimate follows from an application of Lemma \ref{averagelemma} on the function $A(n,z_1,z_2)$. Note that \eqref{quadraticinequality} ensures that the hypothesis in the lemma is satisfied.
For the interval ${x+x^{1-c}<n<2x}$, we have \begin{align*}
  &\sum_{x+x^{1-c}<n<2x}|A(n,z_1,z_2)|\min\left(1,\frac{x}{T|x-n|}\right)\\&\ll\frac{x}{T}\sum_{x+x^{1-\alpha}<n<2x}\frac{\sum_{d|n}\left\{\sum_{d'|d}\left(\sigma_0(d')\sum_{d_1|d/d'}\sigma_0(d_1)\sigma_0(\frac{d}{d'd_1})\right)\right\}\sigma_0(n/d)}{n-x} \\&
     \ll \frac{x}{T}\sum_{l\ll \log x}\frac{1}{U}\sum_{\substack{U<n-x<2U\\ U=2^lx^{1-\alpha}}}\frac{\sum_{d|n}\left\{\sum_{d'|d}\left(\sigma_0(d')\sum_{d_1|d/d'}\sigma_0(d_1)\sigma_0(\frac{d}{d'd_1})\right)\right\}\sigma_0(n/d)}{n-x} \\&
    \ll \frac{x}{T}\log^{8}x.\numberthis\label{errordouble2}  
\end{align*}
One obtains similar bounds for $x/2<T<x-x^{1-c}$. Moreover, \[\frac{x^{\alpha}}{T}\sum_{n=1}^{\infty}\frac{|A(n,z_1,z_2)|}{n^{\alpha}}\ll \frac{x}{T}\log^8T\numberthis\label{errordouble3}\]
Finally, from \eqref{errordouble}, \eqref{errordouble2}, and \eqref{errordouble3}, we deduce that
\[R(x;z_1,z_2)\ll\frac{x}{T}\log^{8}T.\numberthis\label{errordouble4} \]
Next, we solve the line integrals using bounds in \eqref{bound}.
we have
 \begin{align*}
&\int_{\lambda}^{\alpha}\int_{T_0/2}^{T_0}\frac{|\zeta_{\mathbb{K}}(\sigma+it)|^4x^{\sigma}}{|\zeta_{\mathbb{K}}(2(\sigma+it)-z_1-z_2)(\sigma+it)|}d\sigma dt\\& \ll_q \int_{T_0/2}^{T_0} \int_{\lambda}^{1/2}t^{4(7-10\sigma)/6}\log^{18}t \frac{x^{\sigma}}{t}d\sigma dt+ \int_{T_0/2}^{T_0}\int_{1/2}^{\alpha}t^{8(1-\sigma)/3}\log^{18}t \frac{x^{\sigma}}{t}d\sigma dt\\&\ll_q\int_{T_0/2}^{T_0}t^{11/3} \log^{18}t\int_{\lambda}^{1/2}\left(\frac{x}{t^{20/3}}\right)^{\sigma}d\sigma dt+\int_{T_0/2}^{T_0}t^{5/3}\log^{18}t\int_{1/2}^{\alpha}\left(\frac{x}{t^{8/3}}\right)^{\sigma}d\sigma dt\\& \ll_q (T_0^{(8-20\lambda)/3}x^{\lambda}+{x})\log^{18}T_0.
 \end{align*}
 If $a_1-a_2>0$, then we have 
 \begin{align*}
     &\int_{\lambda}^{\alpha}\int_{T_0/2}^{T_0}\frac{|\zeta_{\mathbb{K}}(\sigma+it-z_1)|^4x^{\sigma}}{|\zeta_{\mathbb{K}}(2(\sigma+it)-z_1-z_2)(\sigma+it)|}d\sigma dt \\&\ll_q  \int_{T_0/2}^{T_0}\int_{\lambda}^{1/2+a_1}t^{4(7-10\sigma+10a_1)/6}\log^{18}t \frac{x^{\sigma}}{t}d\sigma dt+ \int_{T_0/2}^{T_0}\int_{1/2+a_1}^{1+a_1}t^{\frac{8(1-\sigma+a_1)}{3}}\log^
     {18}t\frac{x^{\sigma}}{t}d\sigma dt\\&+ \int_{T_0/2}^{T_0}\int_{1+a_1}^{\alpha}\log^{10}t \frac{x^{\sigma}}{t}d\sigma dt\\&= \int_{T_0/2}^{T_0}t^{(11+20a_1)/3} \log^{18}t\int_{\lambda}^{\frac{1}{2}+a_1}\left(\frac{x}{t^{20/3}}\right)^{\sigma}d\sigma dt+\int_{T_0/2}^{T_0}t^{\frac{(5+8a_1)}{3}}\log^{18}t\int_{\frac{1}{2}+a_1}^{1+a_1}\left(\frac{x}{t^{\frac{8}{3}}}\right)^{\sigma}d\sigma dt\\&+\int_{T_0/2}^{T_0}\int_{1+a_1}^{\alpha}\log^{10}t \frac{x^{\sigma}}{t}d\sigma dt\\& \ll_q (T_0^{\frac{8+20a_1-20\lambda}{3}}x^{\lambda}+{x^{1+a_1}})\log^{18}T_0+{x}\log^{10}T_0.
 \end{align*}
 Similarly, \begin{align*}
      &\int_{\lambda}^{\alpha}\int_{T_0/2}^{T_0}\frac{|\zeta_{\mathbb{K}}(\sigma+it-z_2)|^4x^{\sigma}}{|\zeta_{\mathbb{K}}(2(\sigma+it)-z_1-z_2)(\sigma+it)|}d\sigma dt \\&\ll_q \int_{T_0/2}^{T_0}\int_{\lambda}^{1+a_2}t^{\frac{8(1-\sigma+a_2)}{3}}\log^{18}t \frac{x^{\sigma}}{t}d\sigma dt+\int_{T_0/2}^{T_0}\int_{1+a_2}^{\alpha}\log^{10}t \frac{x^{\sigma}}{t}d\sigma dt\\
      &= \int_{T_0/2}^{T_0}t^{\frac{(5+8a_2)}{3}}\log^{18}t\int_{\lambda}^{1+a_2}\left(\frac{x}{t^{\frac{8}{3}}}\right)^{\sigma}d\sigma dt+\int_{T_0/2}^{T_0}\int_{1+a_2}^{\alpha}\log^{10}t \frac{x^{\sigma}}{t}d\sigma dt\\& \ll_q {x^{1+a_2}}\log^{18}T_0+{x}\log^{10}T_0.
 \end{align*}
 and \begin{align*}
   \int_{\lambda}^{\alpha}\int_{T_0/2}^{T_0}\frac{|\zeta_{\mathbb{K}}(\sigma+it-z_1-z_2)|^4x^{\sigma}}{|\zeta_{\mathbb{K}}(2(\sigma+it)-z_1-z_2)(\sigma+it)|}d\sigma dt&\ll_q   \int_{T_0/2}^{T_0}\int_{\lambda}^{1+a_1+a_2}t^{\frac{8(1-\sigma+a_1+a_2)}{3}}\log^{18}t\frac{x^{\sigma}}{t}d\sigma dt\\&+\int_{T_0/2}^{T_0}\int_{1+a_1+a_2}^{\alpha}\log^{10}t \frac{x^{\sigma}}{t}d\sigma dt\\&\ll_q {x^{1+a
   _1+a_2}}\log^{18}T_0+{x}\log^{10}T_0.
 \end{align*} Collecting all the above results and substituting in \eqref{holder}, we obtain 
 \[\int_{\lambda}^{\alpha}\int_{T_0/2}^{T_0}f(z_1,z_2,\sigma+it)\frac{x^{\sigma+it}}{\sigma+it}d\sigma dt\ll_q {x}\log^{14}T_0. \]  Next, we choose $T$ such that $T_0/2<T<T_0$, which gives  \[\int_{\lambda}^{\alpha}f(z_1,z_2,\sigma+iT)\frac{x^{\sigma+iT}}{\sigma+iT}d\sigma\ll_q \frac{x}{T}\log^{14}T.\]Integral along the vertical line $[\lambda-iT,\lambda+iT]$ is given by \begin{align*}
     \int_{\lambda-iT}^{\lambda+iT}f(z_1,z_2,s)\frac{x^{s}}{s}ds&\ll_q \int_{-T}^{T}t^{\frac{4-3a_2}{3}}\log^{18}t\frac{x^{\lambda}}{t}dt\\ &\ll_q T^{\frac{4-3a_2}{3}}x^{\lambda}\log^{18}T. \end{align*}
     Putting $T=x^{1/4}$ in the above estimates, we get the required result.
     \subsection{Cubic Number Field}
     If $\mathbb{K}$ is a cubic number field, then like the case for quadratic, we write $A(n,z_1,z_2)$ using \eqref{sigm,adirichlet} and \eqref{cseriesrepresentation1} to obtain
     \begin{align*}
      |A(n,z_1,z_2)|&\le \left |n^{a_1+a_2}\sum_{\mathcal{N}(\mathcal{I})=n}\left(\sum_{\mathcal{I}_1|\mathcal{I}}1\right)^2\right|\\
      &\le \sum_{d|n}\left[\sum_{d'|d}\left\{\sum_{d_1|d'}\sigma_0(d_1)\sum_{d_2|d/d'}\left(\sum_{d_{21}|d_2}\sigma_0(d_{21})\sum_{d_{22}|d/d_2}\sigma_0(d_{22})\right)\right\}\right]\sum_{d''|n/d}\sigma_0(d'').
  \end{align*}
  For a prime $p$, a direct computation of the right-hand side of the above inequality gives the bound
   \[|A(p,z_1,z_2)|\le 13.\]
   Choosing $T=x^c$ where $c$ is fixed real number, then using Lemma \ref{averagelemma}, we have
   \begin{align*}
       \sum_{|x-n|<x^{1-c}}|A(n,z_1,z_2)|\min\left(1,\frac{x}{T|x-n|}\right)&=\sum_{|x-n|<x^{1-c}}|A(n,z_1,z_2)|\ll \frac{x}{T}\log^{12}T,\numberthis\label{cerrordouble}\end{align*}
The above bounds also hold true for the intervals:  $x/2<T<x-x^{1-c}$ and $x+x^{1-c}<n<2x$. 
Moreover, \[\frac{x^{\alpha}}{T}\sum_{n=1}^{\infty}\frac{|A(n,z_1,z_2)|}{n^{\alpha}}\ll \frac{x}{T}\log^{12}T.
\numberthis\label{cerrordouble3}\]
These estimates yield
\[R(x;z_1,z_2)\ll\frac{x}{T}\log^{12}T.\numberthis\label{cerrordouble4} \]
In the following computations, we employ Dedekind zeta bounds \eqref{cbound} to obtain estimates of the line integrals in \eqref{holder}.
 \begin{align*}
 \int_{\lambda}^{\alpha}\int_{T_0/2}^{T_0}\frac{|\zeta_{\mathbb{K}}(\sigma+it)|^4x^{\sigma}}{|\zeta_{\mathbb{K}}(2(\sigma+it)-z_1-z_2)(\sigma+it)|}d\sigma dt& \ll_q  \int_{T_0/2}^{T_0}\int_{\lambda}^{1/2}t^{4(11-17\sigma)/6}\log^{31}t \frac{x^{\sigma}}{t}d\sigma dt\\&+ \int_{T_0/2}^{T_0}\int_{1/2}^{\alpha}t^{10(1-\sigma)/3}\log^{31}t \frac{x^{\sigma}}{t}d\sigma dt
 \\
 &\ll_q (T_0^{(16-34\lambda)/3}x^{\lambda}+{x})\log^{31}T_0.
 \end{align*}
 If $a_1-a_2>0$, then we have 
 \begin{align*}
    \int_{\lambda}^{\alpha}\int_{T_0/2}^{T_0}\frac{|\zeta_{\mathbb{K}}(\sigma+it-z_1)|^4x^{\sigma}}{|\zeta_{\mathbb{K}}(2(\sigma+it)-z_1-z_2)(\sigma+it)|}d\sigma dt &\ll_q  \int_{T_0/2}^{T_0}\int_{\lambda}^{1/2+a_1}t^{4(11-17\sigma+17a_1)/6}\log^{31}t \frac{x^{\sigma}}{T}d\sigma dt\\&+ \int_{T_0/2}^{T_0}\int_{1/2+a_1}^{1+a_1}t^{\frac{10(1-\sigma+a_1)}{3}}\log^{31}t \frac{x^{\sigma}}{t}d\sigma dt+ \\&\int_{T_0/2}^{T_0}\int_{1+a_1}^{\alpha}\log^{15}t \frac{x^{\sigma}}{t}d\sigma dt\\&
    \ll_q (T_0^{\frac{16+34a_1-34\lambda}{3}}x^{\lambda}+{x^{1+a_1}})\log^{31}T_0+{x}\log^{15}T_0.
 \end{align*}
 Similarly, \begin{align*}
     \int_{\lambda}^{\alpha}\int_{T_0/2}^{T_0}\frac{|\zeta_{\mathbb{K}}(\sigma+it-z_2)|^4x^{\sigma}}{|\zeta_{\mathbb{K}}(2(\sigma+it)-z_1-z_2)(\sigma+it)|}d\sigma dt&\ll_q \int_{T_0/2}^{T_0}\int_{\lambda}^{1+a_2}t^{\frac{10(1-\sigma+a_2)}{3}}\log^{31}t \frac{x^{\sigma}}{t}d\sigma dt\\&+\int_{T_0/2}^{T_0}\int_{1+a_2}^{\alpha}\log^{15}t \frac{x^{\sigma}}{t}d\sigma dt\\
      &
      \ll_q {x^{1+a_2}}\log^{31}T_0+{x}\log^{15}T_0,
 \end{align*}
 and \begin{align*}
   \int_{\lambda}^{\alpha}\int_{T_0/2}^{T_0}\frac{|\zeta_{\mathbb{K}}(\sigma+it-z_1-z_2)|^4x^{\sigma}}{|\zeta_{\mathbb{K}}(2(\sigma+it)-z_1-z_2)(\sigma+it)|}d\sigma dt&\ll_q   \int_{T_0/2}^{T_0}\int_{\lambda}^{1+a_1+a_2}t^{\frac{10(1-\sigma+a_1+a_2)}{3}}\log^{31}t \frac{x^{\sigma}}{t}d\sigma dt\\&+\int_{T_0/2}^{T_0}\int_{1+a_1+a_2}^{\alpha}\log^{15}t \frac{x^{\sigma}}{t}d\sigma dt\\&\ll_q {x^{1+a
   _1+a_2}}\log^{31}T_0+{x}\log^{15}T_0.
 \end{align*}
 The above estimates show that the double integral in \eqref{holder} is bounded by \[\int_{\lambda}^{\alpha}\int_{T_0/2}^{T_0}f(z_1,z_2,\sigma+it)\frac{x^{\sigma+it}}{\sigma+it}d\sigma dt\ll_q {x}\log^{19}T_0.\]
 Next, we choose a $T$ such that $T_0/2<T<T_0$, which gives \[\int_{\lambda}^{\alpha}f(z_1,z_2,\sigma+it)\frac{x^{\sigma+iT}}{\sigma+iT}d\sigma\ll_q \frac{x}{T}\log^{19}T. \]
 Finally, the integral along the vertical line $[\lambda-iT,\lambda+iT]$ is estimated as \begin{align*}
     \int_{\lambda-iT}^{\lambda+iT}f(z_1,z_2,s)\frac{x^{s}}{s}ds&\ll_q \int_{-T}^{T}t^{\frac{5-6a_2}{3}}\log^{31}t\frac{x^{\lambda}}{t}dt\ll_q T^{\frac{5-6a_2}{3}}x^{\lambda}\log^{31}T. \end{align*}
     Putting $T=x^{3/16}$ in the above bounds, we get the required result.
\end{proof}
\section{second moment} \label{sec3}
Our arguments for asymptotics of second moments for quadratic and cubic fields follow ideas from \cite{MR2869206}, and \cite{MR3600410} with several adaptations required to extend the proof for number field. 
\subsection{Second Moment for Quadratic Number Field}
 \begin{proof}[Proof of Theorem \ref{theorem2}]
Let \[\beta_i=1+\frac{i}{\log y},\] where $i\in \{1,2\}.$ Using Lemma \ref{parronlemma}, we have \[\sum_{0<\mathcal{N}(\mathcal{J})\le x}C_{\mathcal{J}}({\mathcal{I}})=\frac{1}{2\pi i}\int_{\beta_i-iT}^{\beta_i+iT}\frac{\sigma_{\mathbb{K},(1-s)}(\mathcal{I})}{\zeta_{\mathbb{K}}(s)}\frac{x^s}{s}ds+\BigO{\frac{x\log y}{T}\sigma_{\mathbb{K},0}(\mathcal{I})}.\]
Squaring the two sides yields
\[\left(\sum_{0<\mathcal{N}(\mathcal{J})\le x}C_{\mathcal{J}}({\mathcal{I}})\right)^2=\frac{1}{(2\pi i)^2}\int_{\beta_1-iT}^{\beta_1+iT}\int_{\beta_2-iT}^{\beta_2+iT}\frac{\sigma_{\mathbb{K},(1-s_1)}(\mathcal{I})\sigma_{\mathbb{K},(1-s_2)}(\mathcal{I})}{\zeta_{\mathbb{K}}(s_1)\zeta_{\mathbb{K}}(s_2)}\frac{x^{s_1+s_2}}{s_1s_2}ds_1ds_2+R(x,\mathcal{I}),\numberthis\label{squaresum}\]
where \begin{align*}
    R(x,\mathcal{I})&\ll \frac{x\log y}{T}\sigma_{\mathbb{K},0}(\mathcal{I})\int_{\beta_i-iT}^{\beta_i+iT}\frac{\sigma_{\mathbb{K},(1-s)}(\mathcal{I})}{\zeta_{\mathbb{K}}(s)}\frac{x^s}{s}ds+ \frac{x^2\log^2 y}{T^2}(\sigma_{\mathbb{K},0}(\mathcal{I}))^2\\
&\ll \frac{x^2\log y\log^3T}{T}(\sigma_{\mathbb{K},0}(\mathcal{I}))^2.\numberthis\label{error7}
\end{align*}
Inserting \eqref{error7} in \eqref{squaresum}, and summing both sides over ideals $\mathcal{I}$ with $\mathcal{N}(\mathcal{I})\le y$, we have \begin{align*}
   \sum_{0<\mathcal{N}(\mathcal{I})\le y} \left(\sum_{0<\mathcal{N}(\mathcal{J})\le x}C_{\mathcal{J}}({\mathcal{I}})\right)^2&=\frac{1}{(2\pi i)^2}\int_{\beta_1-iT}^{\beta_1+iT}\int_{\beta_2-iT}^{\beta_2+iT}\frac{G(s_1,s_2,y)}{\zeta_{\mathbb{K}}(s_1)\zeta_{\mathbb{K}}(s_2)}\frac{x^{s_1+s_2}}{s_1s_2}ds_1ds_2+\\&\BigO{\frac{x^2\log y\log^3T}{T} \sum_{0<\mathcal{N}(\mathcal{I})\le y}(\sigma_{\mathbb{K},0}(\mathcal{I}))^2}\\
   &=I+\BigO{\frac{x^2\log y\log^3T}{T} \sum_{0<\mathcal{N}(\mathcal{I})\le y}(\sigma_{\mathbb{K},0}(\mathcal{I}))^2},\numberthis\label{secondmoment}
\end{align*}
where $G(s_1,s_2,y):=\sum_{0<\mathcal{N}(\mathcal{I})\le y}\sigma_{\mathbb{K},(1-s_1)}(\mathcal{I}) \sigma_{\mathbb{K},(1-s_2)}(\mathcal{I}).$
We take $a_1=a_2=0$ in \eqref{quadraticinequality}, and then use Lemma \ref{averagelemma} to get
\[\sum_{0<\mathcal{N}(\mathcal{I})\le y}(\sigma_{\mathbb{K},0}(\mathcal{I}))^2,\ll y\log^9y.\]
Using Lemma \ref{lemmasix}, the integral $I$ in \eqref{secondmoment} can be written as
\begin{align*}
    I=I_1+I_2+I_3+I_4+\BigOq{x^2y^{5/6}\log^{24}T},\numberthis\label{integral}
\end{align*}
where $I_1, I_2, I_3,$ and $I_4$ are the integrals corresponding to the four terms appearing in $R_0$ in Lemma \ref{lemmasix}. We compute each of them separately below. 
 \subsubsection{Evaluation of $I_1$}
 In order to evaluate the integral $I_1$
 \begin{align*}
    I_1&=\frac{\rho y}{(2\pi i)^2}\int_{\beta_1-iT}^{\beta_1+iT}\int_{\beta_2-iT}^{\beta_2+iT} \frac{\zeta_{\mathbb{K}}(s_1+s_2-1)}{\zeta_{\mathbb{K}}(s_1+s_2)}\frac{x^{s_1+s_2}}{s_1s_2}ds_1ds_2,
\end{align*}
 we shift the line integral in $s_2$-plane to a rectangular contour containing the lines $[\beta_2+iT,1/2+iT]$, $[1/2+iT,1/2-iT]$, $[1/2-iT,\beta_2-iT]$, and $[\beta_2+iT,\beta_2-iT]$. We see that $s_2=2-s_1$ is the pole of the integrand inside the contour and the residue at pole is \[\dfrac{\rho x^2}{\zeta_{\mathbb{K}}(2)s_1(2-s_1)}.\] If $L_{1,1}$, $L_{1,2}$, and $L_{1,3}$ are the integrals along the lines   $[3/2-\beta_1+iT,3/2-\beta_1-iT]$, $[\beta_2+iT,\beta_2-iT]$, and $[3/2-\beta_1-iT,\beta_2-iT]$ respectively, then 
 \begin{align*}
     |L_{1,1}|,|L_{1,3}| &\ll_q y\int_{-T}^{T}\left(\int_{1/2}^{\beta_2}T^{2(2-\beta_1-\sigma)/3}\log^8T \frac{x^{\sigma}}{T}d\sigma\right)\frac{1}{1+|t|}dt\\
     &\ll_q \frac{xy\log^8T}{T^{1/3}}\int_{-T}^{T}\left(\int_{1/2}^{\beta_2}\left (\frac{x}{T^{2/3}}\right)^{\sigma}d\sigma\right)\frac{1}{1+|t|}dt\\
     & \ll_q \frac{x^2y\log^9T}{T}+ \frac{x^{3/2}y\log^9T}{T^{2/3}}.\numberthis\label{l1integral}
 \end{align*}
 Furthermore, the integral along the vertical line is given by \begin{align*}
     |L_{1,2}|&\ll_q yx^{3/2}\int_{-T}^{T}\int_{-T}^{T}\frac{|\zeta_{\mathbb{K}}(\beta_1-1/2+i(t_1+t_2))|}{|\zeta_{\mathbb{K}}(\beta_1+1/2+i(t_1+t_2))|}\frac{1}{(1+|t_1|)(1+|t_2|)}dt_1dt_2\\
     &\ll_q yx^{3/2}\int_{-2T}^{2T}\frac{|\zeta_{\mathbb{K}}(\beta_1-1/2+it)|}{|\zeta_{\mathbb{K}}(\beta_1+1/2+it)|} \int_{-T}^{T}\frac{1}{(1+|t_1|)(1+|t-t_1|)}dt_1dt\\
     & \ll_q yx^{3/2}\log T \int_{-2T}^{2T}\frac{t^{1/3}\log^6t}{(1+|t|)}dt\\ &\ll_q yx^{3/2}T^{1/3}\log^7 T . \numberthis\label{l2integral}
 \end{align*}
 Therefore, we have
 \begin{align*}
     I_1&= \frac{\rho^2yx^2}{\zeta_{\mathbb{K}}(2)}\frac{1}{2\pi i}\int_{\beta_1-iT}^{\beta_1+iT}\frac{1}{s
     _1(2-s_1)}ds_1+\BigOq{\frac{x^2y\log^9T}{T}+ yx^{3/2}T^{1/3}\log^7 T }\\
     &= \frac{\rho^2yx^2}{2\zeta_{\mathbb{K}}(2)}+\BigOq{\frac{x^2y\log^9T}{T}+ yx^{3/2}T^{1/3}\log^7T}.\numberthis\label{I1integral}
     \end{align*}
     \subsubsection{Evaluation of $I_2$}
     The integral $I_2$ is given by \begin{align*}
         I_2 & =\frac{\rho y^2}{ (2\pi i)^2}\int_{\beta_1-iT}^{\beta_1+iT}\int_{\beta_2-iT}^{\beta_2+iT} \frac{\zeta_{\mathbb{K}}(2-s_1)\zeta_{\mathbb{K}}(1-s_1+s_2)}{(2-s_1)\zeta_{\mathbb{K}}(2-s_1+s_2)\zeta_{\mathbb{K}}(s_1)}\frac{x^{s_1+s_2}}{y^{s_1}s_1s_2}ds_1ds_2.
     \end{align*}
     We move the line integral over $s_1$-plane to a contour with vertices $\beta_1+iT$, $2-\frac{3}{\log y}+iT$, $2-\frac{3}{\log y}-iT$, and $\beta_1+iT$. The integrand has a simple at $s_1=s_2$ with residue \[\rho\frac{\zeta_{\mathbb{K}}(2-s_2)x^{2s_2}/y^{s_2}}{s_2^2(2-s_2)\zeta_{\mathbb{K}}(2)\zeta_{\mathbb{K}}(s_2)}\] 
     Let $L_{2,1}$, $L_{2,3}$ be the integrals along the horizontal lines of the contour, and $L_{2,2}$ is the integral along the vertical line. Thus, we have 
     \begin{align*}
         |L_{2,1}|,|L_{2,3}|&\ll y^2 \int_{-T}^{T}\left(\int_{\beta_1}^{2-\frac{3}{\log y}}\frac{\zeta_{\mathbb{K}}(2-\sigma)\zeta_{\mathbb{K}}(1-\sigma+\beta_2)}{\zeta_{\mathbb{K}}(2-\sigma+\beta_2)\zeta_{\mathbb{K}}(\sigma)}\frac{x^{\sigma+\beta_2}}{y^{\sigma}T^2}d\sigma\right)\frac{1}{1+|t|}dt\\
         &\ll_q \frac{xy^2}{T^2}\int_{-T}^{T}\left(\int_{\beta_1}^{3/2}T^{\frac{2-2(2-\sigma)}{3}}T^{\frac{2-2(1-\sigma+\beta_2)}{3}}\log^{10}T\frac{x^{\sigma}}{y^{\sigma}}d\sigma\right)\frac{1}{1+|t|}dt\\
         &\ll_q \frac{xy^2\log^{10}T}{T^{2+4/3}}\int_{-T}^{T}\left(\int_{\beta_1}^{3/2}\left(\frac{xT^{4/3}}{y}\right)^{\sigma}d\sigma\right)\frac{1}{1+|t|}dt\\
         &\ll_q \frac{x^{5/2}y^{1/2}\log^{11}T}{T^{4/3}}+\frac{x^2y\log^{11}T}{T^2},\numberthis\label{l21}
     \end{align*}
     and
      \begin{align*}
         |L_{2,2}|&\ll x^{5/2}y^{1/2}\int_{-T}^{T}\int_{-T}^{T}\frac{|\zeta_{\mathbb{K}}(1/2-it_1)||\zeta_{\mathbb{K}}(-1/2+\beta_2+i(-t_1+t_2))|}{|\zeta_{\mathbb{K}}(3/2+it_1)||\zeta_{\mathbb{K}}(1/2+\beta_2+i(-t_1+t_2))|(1+|t_1|)^2(1+|t_2|)}dt_1dt_2\\
         &\ll_q x^{5/2}y^{1/2}\int_{-2T}^{2T}\frac{|\zeta_{\mathbb{K}}(-1/2+\beta_2+it)|}{|\zeta_{\mathbb{K}}(1/2+\beta_2+it)|}\int_{-T}^{T}\frac{t_1^{1/3}\log ^6t_1}{(1+|t_1|)^2(1+|t_1+t|)}dt_1dt\\
         &\ll_q \frac{x^{5/2}y^{1/2}\log ^6T}{T^{2/3}}\int_{-2T}^{2T}\frac{t^{1/3}\log ^6t}{(1+|t|)}dt\\
         &\ll_q\frac{x^{5/2}y^{1/2}\log ^{12}T}{T^{1/3}}.\numberthis\label{l22}
     \end{align*}
    Combining the above estimates, we have \begin{align*}
         I_2&=\frac{\rho^2 y^2}{\zeta_{\mathbb{K}}(2) (2\pi i)}\int_{\beta_2-iT}^{\beta_2+iT}\frac{\zeta_{\mathbb{K}}(2-s_2)x^{2s_2}/y^{s_2}}{s_2^2(2-s_2)\zeta_{\mathbb{K}}(s_2)}ds_2+\BigOq{ \frac{x^2y\log^{11}T}{T^2}+\frac{x^{5/2}y^{1/2}\log ^{12}T}{T^{1/3}}}\\
         &= -\frac{\rho^2x^4\zeta_{\mathbb{K}}(0)}{4\zeta_{\mathbb{K}}(2)^2 }+\BigOq{ \frac{x^2y\log^{11}T}{T^2}+\frac{x^{5/2}y^{1/2}\log ^{12}T}{T^{1/3}}}.\numberthis\label{I2integral}
     \end{align*}
     \subsubsection{Evaluation of $I_3$}
     The integral $I_3$ is given as \begin{align*}
       I_3 & =\frac{\rho y^2}{ (2\pi i)^2}\int_{\beta_1-iT}^{\beta_1+iT}\int_{\beta_2-iT}^{\beta_2+iT} \frac{\zeta_{\mathbb{K}}(2-s_2)\zeta_{\mathbb{K}}(1-s_2+s_1)}{(2-s_2)\zeta_{\mathbb{K}}(2-s_2+s_1)\zeta_{\mathbb{K}}(s_2)}\frac{x^{s_1+s_2}}{y^{s_2}s_1s_2}ds_1ds_2.  
     \end{align*}
    To estimate the integral $I_3$, we modify the line integration over $s_2$ to the contour containing the vertices $\beta_2+iT$, $\beta_2-iT$, $3/2+iT$, and $3/2-iT$, and denote the integration along the line $[\beta_2+iT,3/2+iT]$, $[3/2+iT,3/2-iT]$, and $[3/2-iT,\beta_2-iT]$ are $L_{3,1}$, $L_{3,2}$, and $L_{3,3}$ respectively. There is no pole of integrand
    inside the contour. So, the integral along the horizontal lines are 
    \begin{align*}
        |L_{3,1}|,|L_{3,3}|& \ll y^2 \int_{-T}^{T}\left(\int_{\beta_2}^{3/2}\frac{\zeta_{\mathbb{K}}(2-\sigma)\zeta_{\mathbb{K}}(1-\sigma+\beta_1)}{\zeta_{\mathbb{K}}(2-\sigma+\beta_1)\zeta_{\mathbb{K}}(\sigma)}\frac{x^{\sigma+\beta_1}}{y^{\sigma}T^2}d\sigma\right)\frac{1}{1+|t|}dt\\
          &\ll_q \frac{xy^2}{T^2}\int_{-T}^{T}\left(\int_{\beta_2}^{3/2}T^{\frac{2-2(2-\sigma)}{3}}T^{\frac{2-2(1-\sigma+\beta_1)}{3}}\log^{10}T\frac{x^{\sigma}}{y^{\sigma}}d\sigma\right)\frac{1}{1+|t|}dt\\
         &\ll_q \frac{xy^2\log^{10}T}{T^{2+4/3}}\int_{-T}^{T}\left(\int_{\beta_2}^{3/2}\left(\frac{xT^{4/3}}{y}\right)^{\sigma}d\sigma\right)\frac{1}{1+|t|}dt\\
         &\ll_q \frac{x^{5/2}y^{1/2}\log^{11}T}{T^{4/3}}+\frac{x^2y\log^{11}T}{T^2}.\numberthis\label{l31}
    \end{align*}
    One can evaluate the integral along vertical line same as \eqref{l22}. Therefore 
    \[|L_{3,2}|\ll_q\frac{x^{5/2}y^{1/2}\log ^{11}T}{T^{1/3}},\numberthis\label{l32}\]
    and \[I_3=\BigOq{\frac{x^2y\log^{11}T}{T^2}+\frac{x^{5/2}y^{1/2}\log ^{12}T}{T^{1/3}}}.\numberthis\label{I3integral}\]
    \subsubsection{Evaluation of $I_4$}
    Finally, the integration $I_4$ is equal to \[I_4=\frac{\rho }{ (2\pi i)^2}\int_{\beta_1-iT}^{\beta_1+iT}\int_{\beta_2-iT}^{\beta_2+iT} \frac{\zeta_{\mathbb{K}}(2-s_1)\zeta_{\mathbb{K}}(2-s_2)\zeta_{\mathbb{K}}(3-s_1-s_2)}{(3-s_1-s_2)\zeta_{\mathbb{K}}(4-s_1-s_2)\zeta_{\mathbb{K}}(s_1)\zeta_{\mathbb{K}}(s_2)}\frac{y^{3-s_1-s_2}x^{s_1+s_2}}{s_1s_2}ds_1ds_2.\]
    We estimate $I_4$ by shifting the integration over $s_2$ to the contour with vertices $\beta_2+iT$, $\beta_2-iT$, $5/2-\beta_1+iT$, and $5/2-\beta_1-iT$. Suppose the integration along the line $[\beta_2+iT,5/2-\beta_1+iT]$, $[5/2-\beta_1+iT,5/2-\beta_1-iT]$, and $[5/2-\beta_1-iT,\beta_2-iT]$ are $L_{4,1}$, $L_{4,2}$, and $L_{4,3}$ respectively. Then 
    \begin{align*}
        |L_{4,1}|,|L_{4,3}|&\ll  y^2x \int_{-T}^{T}\left(\int_{\beta_2}^{5/2-\beta_1}\frac{\zeta_{\mathbb{K}}(2-\beta_1)\zeta_{\mathbb{K}}(2-\sigma)\zeta_{\mathbb{K}}(3-\beta_1-\sigma)}{\zeta_{\mathbb{K}}(4-\sigma-\beta_1)\zeta_{\mathbb{K}}(\beta_1)\zeta_{\mathbb{K}}(\sigma)}\frac{x^{\sigma}}{y^{\sigma}T^2}d\sigma\right)\frac{1}{1+|t|}dt\\
        &\ll_q \frac{xy^2}{T^2}\int_{-T}^{T}\left(\int_{\beta_1}^{5/2-\beta_1}T^{\frac{2-2(2-\sigma)}{3}}T^{\frac{2-2(1-\sigma-\beta_1)}{3}}\log^{14}T\frac{x^{\sigma}}{y^{\sigma}}d\sigma\right)\frac{1}{1+|t|}dt\\
        &\ll_q \frac{x^{5/2}y^{1/2}\log^{15}T}{T^{4/3}}+\frac{x^2y\log^{15}T}{T^2}.\numberthis\label{l41}
    \end{align*}
    And \begin{align*}
        &|L_{4,2}|\\& \ll x^{5/2}y^{1/2}\int_{-T}^{T}\int_{-T}^{T}\frac{|\zeta_{\mathbb{K}}(2-\beta_1-it_1)||\zeta_{\mathbb{K}}(-1/2+\beta_1-it_1)||\zeta_{\mathbb{K}}(1/2+i(-t_1-t_2))|dt_1dt_2}{|\zeta_{\mathbb{K}}(3/2+i(-t_1-t_2))||\zeta_{\mathbb{K}}(\beta_1+it_1)||\zeta_{\mathbb{K}}(5/2-\beta_1+it_2)|(1+|t_1|)^2(1+|t_2|)}\\
         &\ll_q x^{5/2}y^{1/2}\int_{-2T}^{2T}\frac{|\zeta_{\mathbb{K}}(1/2+it)|}{|\zeta_{\mathbb{K}}(3/2+\beta_2+it)|}\int_{-T}^{T}\frac{t_1^{1/3}\log ^{12}t_1}{(1+|t_1|)^2(1+|t-t_1|)}dt_1dt\\
         &\ll_q \frac{x^{5/2}y^{1/2}\log ^{12}T}{T^{2/3}}\int_{-2T}^{2T}\frac{t^{1/3}\log ^6t}{(1+|t|)}dt\\
         &\ll_q\frac{x^{5/2}y^{1/2}\log ^{18}T}{T^{1/3}}.\numberthis\label{l42} 
    \end{align*}
    Thus, \[I_4=\BigOq{\frac{x^2y\log^{15}T}{T^2}+\frac{x^{5/2}y^{1/2}\log ^{18}T}{T^{1/3}}}.\numberthis\label{I4integral}\]
Collecting the results from \eqref{I1integral}, \eqref{I2integral}, \eqref{I3integral}, \eqref{I4integral} and inserting in \eqref{integral}, we deduce
\begin{align*}
    I&=\frac{\rho^2yx^2}{\zeta_{\mathbb{K}}(2)}-\frac{\rho^2x^4\zeta_{\mathbb{K}}(0)}{4\zeta_{\mathbb{K}}(2)^2 }
   +\BigOq{\frac{x^2y\log^9T}{T}+ yx^{3/2}T^{1/3}\log^7T}\\&+\BigOq{\frac{x^{5/2}y^{1/2}\log ^{18}T}{T^{1/3}}+x^2y^{5/6}\log^{24}T}.\end{align*}
   Take $T=x^{3/2-\epsilon}$, then for $x\le y <x^2$, 
   \begin{align*}
        \sum_{0<\mathcal{N}(\mathcal{I})\le y} \left(\sum_{0<\mathcal{N}(\mathcal{J})\le x}C_{\mathcal{J}}({\mathcal{I}})\right)^2&= \frac{\rho^2yx^2}{2\zeta_{\mathbb{K}}(2)}-\frac{\rho^2x^4\zeta_{\mathbb{K}}(0)}{4\zeta_{\mathbb{K}}(2)^2 }+\BigOq{yx^{2-\epsilon}\log^7x+x^2y^{5/6}\log^{24}x},
   \end{align*}
   and for $x^2\le y <x^3$,
    \begin{align*}
        \sum_{0<\mathcal{N}(\mathcal{I})\le y} \left(\sum_{0<\mathcal{N}(\mathcal{J})\le x}C_{\mathcal{J}}({\mathcal{I}})\right)^2&= \frac{\rho^2yx^2}{2\zeta_{\mathbb{K}}(2)}+\BigOq{yx^{2-\epsilon}\log^7x +x^2y^{5/6}\log^{24}x}.
   \end{align*}
   \end{proof}
   \subsection{Second Moment for Cubic Number Field}
\begin{proof}[Proof of Theorem \ref{theorem4}]
   The proof uses the same steps as in the degree two case, except several technical changes which arise due to the difference between the bounds of the Dedekind zeta function for cubic \eqref{cbound}, and quadratic number fields \eqref{bound}
   In here, the optimal choice of $T$ appearing upon truncation of the infinite line is $T=x^{5/6-\epsilon}$ for a fixed $\epsilon>0$. Since the arguments do not change, we will omit a detailed illustration. 
   \end{proof}
 \section{Proof of Theorem \ref{theorem5}} \label{sec4}
 In this section, using elementary techniques, we prove the second moment for number fields of any degree satisfying condition \eqref{property1}. As the degree of a number field increases, due to large bounds for the associated Dedekind zeta function in the required regions, the error terms originating from the line integrals in Perron's formula dominate over the main terms. Consequently, we avoid an analytic approach for higher degree number 
 fields at the cost of losing a second main term. 
 \begin{proof}[Proof of Theorem \ref{theorem5}]
 From \eqref{one}, we have 
 \begin{align*}
      \sum_{0<\mathcal{N}(\mathcal{I})\le y} \left(\sum_{0<\mathcal{N}(\mathcal{J})\le x}C_{\mathcal{J}}({\mathcal{I}})\right)^2&=  \sum_{0<\mathcal{N}(\mathcal{I})\le y} \left(\sum_{0<\mathcal{N}(\mathcal{J})\le x}\sum_{\substack{\mathcal{I}_1|\mathcal{J}\\\mathcal{I}_1|\mathcal{I}}}\mathcal{N}(\mathcal{I}_1)\mu(\frac{\mathcal{J}}{\mathcal{I}_1})\right)^2\\
      &= \sum_{0<\mathcal{N}\mathcal{(I}\mathcal{J})\le x}\sum_{0<\mathcal{N}\mathcal{(I'}\mathcal{J'})\le x}\mathcal{N}(\mathcal{I})\mathcal{N}(\mathcal{I}')\mu(\mathcal{J})\mu(\mathcal{J}')\sum_{\substack{0<\mathcal{N}\mathcal{I})\le y\\\mathcal{I}|\mathcal{I}_1\\\mathcal{I}'|\mathcal{I}_1}}1 \numberthis\label{gennumb}
      \end{align*}
      From the hypothesis in Theorem \ref{theorem5}, the innermost sum is given by
      \[\#\{I: I\subseteq O_{\mathbb{K}}: \mathcal{N}(\mathcal{I})\le y, \  \mathcal{I}|\mathcal{I}_1, \mathcal{I}'|\mathcal{I}_1 \}=\frac{\rho_{\mathbb{K}} y}{\mathcal{N}(\mathcal{I}\cap\mathcal{I}')}+\BigO{\left(\frac{y}{\mathcal{N}(\mathcal{I}\cap\mathcal{I}')}\right)^{\alpha}}. \]
      Using this, the left hand side of \eqref{gennumb} equals
      \begin{align*}
      \\
      &\rho_{\mathbb{K}} y\sum_{0<\mathcal{N}\mathcal{(I}\mathcal{J})\le x}\sum_{0<\mathcal{N}(\mathcal{I'}\mathcal{J'})\le x}\mathcal{N}(\mathcal{I}+\mathcal{I}')\mu(\mathcal{J})\mu(\mathcal{J}')\\&+\BigO{y^{\alpha}\sum_{0<\mathcal{N}\mathcal{(I}\mathcal{J})\le x}\sum_{0<\mathcal{N}\mathcal{(I'}\mathcal{J'})\le x}\mathcal{N}(\mathcal{I}+\mathcal{I}')^{\alpha}\mathcal{N}\mathcal{(I})^{1-\alpha}\mathcal{N}\mathcal{(I'})^{1-\alpha}}\\
      &=:I_1+I_2\numberthis\label{prufer1}
 \end{align*}
 
     Let $\gcd(\mathcal{I},\mathcal{I}')=\mathcal{A}$, then $\mathcal{I}=\mathcal{A}\mathcal{E}$, and $\mathcal{I}'=\mathcal{A}\mathcal{E}'$ such that $\mathcal{E}+\mathcal{E}'=\mathcal{O}_{\mathbb{K}}$. This yields  
     \begin{align*}
         I_1&=\rho_{\mathbb{K}} y\sum_{0<\mathcal{N}(\mathcal{A}\mathcal{I}\mathcal{J})\le x}\sum_{\substack{0<\mathcal{N}(\mathcal{E'}\mathcal{A}\mathcal{J'})\le x\\\mathcal{E}+\mathcal{E}'=\mathcal{O}_{\mathbb{K}}}}\mathcal{N}\mathcal{(A})\mu(\mathcal{J})\mu(\mathcal{J}')\\
         &=\rho_{\mathbb{K}} y\sum_{0<\mathcal{N}(\mathcal{A}\mathcal{I}\mathcal{J})\le x}\sum_{0<\mathcal{N}(\mathcal{E'}\mathcal{A}\mathcal{J'})\le x}\mathcal{N}\mathcal{(A})\mu(\mathcal{J})\mu(\mathcal{J}')\sum_{\mathcal{M}|\mathcal{E}+\mathcal{E}'}\mu(\mathcal{M})\\
         &= \rho_{\mathbb{K}} y\sum_{0<\mathcal{N}(\mathcal{A}\mathcal{M})\le x}\mathcal{N}\mathcal{(A}\mathcal{M})\left(\sum_{0<\mathcal{N}(\mathcal{E}\mathcal{J})\le x/\mathcal{N}\mathcal{(A}\mathcal{M})}\mu(\mathcal{J})\right)^2\\
         &= \rho_{\mathbb{K}} y\sum_{0<\mathcal{N}(\mathcal{A}\mathcal{M})\le x}\mathcal{N}\mathcal{(A})\mu(\mathcal{M})\\
         &= \frac{\rho_{\mathbb{K}}^2 x^2y}{2\zeta_{\mathbb{K}}(2)}+\BigO{xy\log x},\numberthis\label{prufer2}
     \end{align*}
and \begin{align*}
    I_2&\ll y^{\alpha}\sum_{0<\mathcal{N}(\mathcal{A}\mathcal{I}\mathcal{J})\le x}\sum_{\substack{0<\mathcal{N}(\mathcal{E'}\mathcal{A}\mathcal{J'})\le x\\\mathcal{E}+\mathcal{E}'=\mathcal{O}_{\mathbb{K}}}}\mathcal{N}\mathcal{(A})^{\alpha}\mathcal{N}(\mathcal{E}\mathcal{A})^{1-\alpha}\mathcal{N}(\mathcal{E'}\mathcal{A})^{1-\alpha}\\
    &\ll y^{\alpha} \sum_{0<\mathcal{N}(\mathcal{A}\mathcal{I}\mathcal{J})\le x}\sum_{0<\mathcal{N}(\mathcal{E'}\mathcal{A}\mathcal{J'})\le x}\mathcal{N}\mathcal{(A})^{2-\alpha}\mathcal{N}(\mathcal{E})^{1-\alpha}\mathcal{N}(\mathcal{E'})^{1-\alpha}\sum_{\mathcal{M}|\mathcal{E}+\mathcal{E}'}\mu(\mathcal{M})\\
    &\ll y^{\alpha} \sum_{0<\mathcal{N}(\mathcal{A}\mathcal{M})\le x}\mathcal{N}\mathcal{(A})^{2-\alpha}\left(\sum_{0<\mathcal{N}(\mathcal{E}\mathcal{J})\le x/\mathcal{N}\mathcal{(A}\mathcal{M})}\mathcal{N}(\mathcal{E})^{1-\alpha}\right)^2\\
    &\ll y^{\alpha} \sum_{0<\mathcal{N}(\mathcal{A}\mathcal{M})\le x}\mathcal{N}\mathcal{(A})^{2-\alpha}\frac{x^{(2-\alpha)^2}}{\mathcal{N}(\mathcal{A}\mathcal{M})^{(2-\alpha)^2}}
    \ll y^{\alpha}x^{3-\alpha}.\numberthis\label{prufer3}
\end{align*}
On substitution of \eqref{prufer2} and \eqref{prufer3} in \eqref{prufer1}, we obtain the required result.
\end{proof}
\bibliographystyle{plain} 
\bibliography{reference}
\end{document}